\documentclass[11pt]{amsart}
\pdfoutput=1
\usepackage{amssymb}
\usepackage{amsmath}
\usepackage{amsfonts}
\usepackage[usenames]{color}
\usepackage{graphicx}
\usepackage{array}
\usepackage{psfrag}
\usepackage{color}
\usepackage{ulem}

\makeatletter
 
 \@addtoreset{equation}{section}
\makeatother

\newtheorem{theorem}{Theorem}

\theoremstyle{definition}

\theoremstyle{remark}

\numberwithin{equation}{section}

\newcommand{\R}{\mathbb{R}}

\newcommand{\dfn}[1]{\textit{#1}}

\begin{document}

\title{The extremal function for apex graphs}
\author{Elena Pavelescu}
\address{Department of Mathematics and Statistics, University of South Alabama, Mobile, AL, 36688,  \textit{elenapavelescu@southalabama.edu}}
\date{\today}
\maketitle

\begin{abstract}
McCarty and Thomas conjectured that a linklessly embeddable graph with $n\ge 7 $ vertices and $t$ triangles has at most $3n-9 +\frac{t}{3}$ edges.
Thomas and Yoo proved this to be true for apex graphs.
We give a  shorter  and simpler proof for the apex case.\end{abstract}
\vspace{0.2in}

\section{Introduction}

All graphs in this paper are finite and simple.
A graph is \dfn{linklessly embeddable} if it can be embedded in $\R^3$ (or, equivalently, $S^3$) in such a way that no two cycles of the graph are linked. 
In \cite{MT}, McCarty and Thomas proved that a bipartite linklessly embeddable graph on $n\ge 5$ vertices has at most
$3n-10$ edges, unless it is isomorphic to the complete bipartite graph $K_{3,n-3}$.
They conjectured that a linklessly embeddable graph with $n\ge 7 $ vertices and $t$ triangles has at most $3n-9 +\frac{t}{3}$ edges.
Thomas and Yoo proved the conjecture holds for apex graphs \cite{TY}. 
A graph $G$ is apex if it can be made planar by deleting one vertex.
Apex graphs are know to be  linklessly embeddable \cite{Sa}.
Here we give a simpler and shorter proof for the apex case.
We use Euler's identity for connected planar graphs: $v-e+f=2$, where $v$, $e$, and $f$ represent the number of vertices, edges, and faces of a plane embedding.

\section{Main Result}

\begin{theorem}
Every apex graph with $n\ge 7$ vertices and $t$ triangles has at most $3n-9 +\frac{t}{3}$ edges.
\end{theorem}

\begin{proof}
We first show the theorem is true for maximal apex graphs. 
To construct a maximal apex graph $G$, start with a maximal planar graph $H$ and add one vertex $v$ together with all the edges from $v$ to the vertices of $H$. 
We write $G=H*v$.
If $G$ has $n$ vertices, then $G$ has $4n-10$ edges: $n-1$ edges incident to $v$, and $3n-9$ edges which are edges of $H$.
The graph $G$ has $5n-15+L$ triangles: $3n-9$ triangles containing $v$ and  $2n-6$ triangles which are the faces of a planar embedding of $H$, by Euler's identity, and $L$ triangles of $H$ which are not faces of a planar embedding of $H$.
Since $4n-10 \le 3n-9 + \frac{5n-15+L}{3}$, for $n\ge 7$, $L\ge 0$, the theorem is true for maximal apex graphs.

Let  $G$ be an apex graph with $n$ vertices. 
Then $G$ is obtained from a maximal apex graphs $H* v$ through removal of edges.
Assume $G$ is obtained by removing $k$ edges incident to $v$ and $s$ edges of $H$. 
Let $a_1, a_2, \ldots a_k \in V(H)$ and consider the graph $G'$ obtained from $H*v$ by removing edges  $va_1, va_2, \ldots, va_k.$
Then \[t(H*v) - t(G') = |E(H)| -  |E(H-\{a_1, a_2, \ldots , a_k\} )|, \tag{1}\]
since the triangles of $H*v$ containing the edge $va_i$ correspond to the edges of $H$ incident to the vertex  $a_i$.
We claim that for $k\ge 1$ and an arbitrary choice of  vertices $a_1, a_2, \ldots a_k \in V(H)$, $$|E(H)| -  |E(H-\{a_1, a_2, \ldots , a_k\} )|\le 2n+3k-12, $$
unless $k=2$, and $H$ has  at least $2n-5$ edges incident to $a_1$, $a_2$, or both. This exceptional case is discussed at the very end of the proof.

\noindent To prove the claim, we consider $k=1$, $k=2$, and $k\ge 3$ separately.\\
For $k=1$, 
$$|E(H)| -  |E(H-\{a_1\} )| =\deg_H(a_1)\le n-2 \le 2n+3-12, \textrm{ for } n\ge 7.$$
For $k=2$ in non-exceptional cases, 
$$|E(H)| -  |E(H-\{a_1, a_2\} )| \le  2n-6 = 2n+6-12.$$

\noindent For $k\ge 3$, the graph $H$ has at most  $3k-6$ edges $a_ia_j$, with $1\le i<j\le k$ (by Euler's identity), and at most $2(n-1)-4 =2n-6$ edges  $a_ib_j$, with $1\le i\le k$ and $1\le j\le n-1-k$ (by Euler's identity for bipartite graphs).
Then the claim follows, since  $$|E(H-\{a_1, a_2, \ldots , a_k\} )| \ge |E(H)| - (3k-6) - (2n-6) =  |E(H)| - (2n+3k-12).$$ 
This  last inequality together with (1) gives \[t(H*v) - t(G') \le 2n+3k-12, \tag{2}\] meaning that when deleting $k$ edges incident to $v$ from $H*v$, at most $2n+3k-12$ triangles are deleted.

To obtain $G$, delete $s$ edges of $G'$ which are also edges of $H$. 
Each such edge belongs to at most one triangle containing the vertex $v$, at most two triangles which are faces of $H$, and possibly to some non-face triangles of $H$. 
These $s$ edges of $H$ belong to at most $L$ non-face triangles of $H$.
These observations yield \[t(G') \le t(G) +3s+L.\tag{3}\]
Adding the  inequalities (2) and (3) gives 
$$ 5n-15+L=t(H*v)\le t(G) +2n+3k+3s+L-12, $$
and further,
$$3n-3k-3s-3\le t(G).$$
\noindent We check  the inequality holds for the apex graph $G$:
\[ \tag{4}|E(G)| = 4n-10 - k - s = (3n-9) + (n -k-s-1) = \] $$=3n-9+\frac{3n-3k-3s-3}{3} \le 3n-9+\frac{t(G)}{3}. $$ 
\textit{The exceptional case.} Consider  $a_1, a_2\in V(H)$ such that  $H$ has at least $2n-5$ edges incident to $a_1$, $a_2$, or both.
This is only possible if $\deg_H(a_1)=\deg_H(a_2)=n-2$, that is, $a_1$ and $a_2$ are both adjacent to all other vertices of $H$ and $a_1a_2\in E(H)$. 
The graph $H$ is isomorphic to the join $K_2 + P_{n-3}$, where $P_{n-3}$ denotes the path with $n-3$ vertices.
Exactly $2n-5$ edges of $H$ are incident to $a_1$, $a_2$, or both.
For $G'=H*v - \{va_1, va_2\}$, \[\tag{5}t(H*v)-t(G') = 2n-5.\]
To obtain $G$, delete $s$ edges of $G'$ which are also edges of $H$. 
Edge $a_1a_2$ belongs to $L+2$ triangles of $G'$. 
Edges $a_1b_i$  (and $a_2b_i$), $i=1, 2, \ldots, n-3$,  belong to at most three triangles of $G'$: two triangles which are faces in $H$ and at most one non-face triangle of $H$.
Edges $b_ib_j$, $1\le i<j\le n-3$, belong to three triangles of $G'$: two triangles which are faces in $H$ and one triangle containing the vertex $v$.
These observations yield ($L\ge 1$ for $n\ge 7$)
\[t(G') \le t(G) +L+2 + 3(s-1) = t(G) + 3s + L -1.\tag{6}\]

Adding the identities (5) and (6) gives
\[  5n-15+L = t(H*v) \le t(G) + 2n +3s +L -6,\]
and further
\[ 3n -3s -9 \le t(G)\]
Since $k=2$, this inequality is equivalent to 
$$3n-3k-3s-3\le t(G),$$
and the same check as in (4) completes the proof.

\end{proof}

\noindent {\bf{Acknowledgements.}} The author would like to thank Youngho Yoo for noticing that non-face triangles were missing in the first version of the proof, and for suggesting a more detailed  explanation for the cases  $k=1$ and $k=2$.

\bibliographystyle{amsplain}

\end{document}